\documentclass[a4paper,11pt]{article}
\usepackage{amsmath,amssymb,amsthm,graphicx}
\usepackage[top=80pt, bottom=90pt, left=70pt, right=70pt]{geometry}
\usepackage{comment}
\usepackage{color}
\def\qed{\hfill $\Box$}

\newcommand{\argmin}{\operatornamewithlimits{argmin}}

\newtheorem{thm}{\bfseries Theorem}[section]
\newtheorem{lem}[thm]{\bfseries Lemma} 
\newtheorem{remark}[thm]{\bfseries Remark} 

\usepackage{bm}

\begin{document}
\title{The quadratic M-convexity testing problem\footnotemark[1]}
\author{Yuni Iwamasa\footnotemark[2]}
\date{\today}
\footnotetext[1]{A preliminary version of this paper has appeared in the proceedings of the 10th Japanese-Hungarian Symposium on Discrete Mathematics and Its Applications (JH 2017).}
\footnotetext[2]{Department of Mathematical Informatics,
Graduate School of Information Science and Technology,
University of Tokyo, Tokyo, 113-8656, Japan.\\
Email: \texttt{yuni\_iwamasa@mist.i.u-tokyo.ac.jp}
}
\maketitle

\begin{abstract}
	M-convex functions,
	which are a generalization of valuated matroids,
	play a central role in discrete convex analysis.
	Quadratic M-convex functions constitute a basic and important subclass of M-convex functions,
	which has a close relationship with phylogenetics as well as valued constraint satisfaction problems.
	In this paper,
	we consider the quadratic M-convexity testing problem (QMCTP),
	which is the problem of deciding whether a given quadratic function on $\{0,1\}^n$ is M-convex.
	We show that QMCTP is co-NP-complete in general,
	but is polynomial-time solvable under a natural assumption.
	Furthermore,
	we propose an $O(n^2)$-time algorithm for solving QMCTP in the polynomial-time solvable case.
\end{abstract}
\begin{quote}
\end{quote}

\section{Introduction}
A function $f$ on $\{0,1\}^n$ is said to be {\it M-convex}~\cite{AM/M96}
if it satisfies the following generalization of matroid exchange axiom:
\begin{description}
	\item[Exchange Axiom:] For $x,y \in \textrm{dom }f$ and $i \in \textrm{supp}(x) \setminus \textrm{supp}(y)$, there exists $j \in {\rm supp}(y) \setminus \textrm{supp}(x)$ such that
	\begin{align*}
	f(x) + f(y) \geq f(x - \chi_i + \chi_j) + f(y + \chi_i - \chi_j),
	\end{align*}
\end{description}
where $\textrm{dom }f := \{ x \in \{0,1\}^n \mid \text{$f(x)$ takes a finite value} \}$ is the effective domain of $f$, $\textrm{supp}(x) := \{ i \mid x_i = 1 \}$ for $x = (x_1,x_2,\dots,x_n) \in \{0,1\}^n$,
and $\chi_i$ is the $i$th unit vector.
In general, M-convex functions are defined on the integer lattice $\mathbf{Z}^n$.
In this paper,
we restrict ourselves to M-convex functions defined on $\{0,1\}^n$, which are equivalent to the negative of {\it valuated matroids} introduced by Dress--Wenzel~\cite{AML/DW90, AM/DW92}.
M-convex functions play a central role in {\it discrete convex analysis}~\cite{book/Murota03}.
Indeed, M-convex functions appear in many areas such as operations research, economics, and game theory (see e.g.,~\cite{book/Murota03, incollection/M09,  JMID/M16}).
Quadratic M-convex functions also appear in many areas,
and constitute a basic and important class of discrete functions.
Quadratic M-convex functions have a close relationship with {\it tree metrics}~\cite{JJIAM/HM04},
which is an important concept for mathematical analysis in phylogenetics~(see e.g., \cite{book/SempleSteel03}).
Recently, Iwamasa--Murota--\v{Z}ivn\'y~\cite{arxiv/IMZ17} have revealed hidden quadratic M-convexity in valued constraint satisfaction problems (VCSPs) with joint winner property~\cite{AI/CZ11},
and presented a perspective to their polynomial-time solvability from discrete convex analysis.

In this paper,
we consider the {\it quadratic M-convexity testing problem (QMCTP)} defined as follows.
Let $\overline{\mathbf{R}} := \mathbf{R} \cup \{ +\infty \}$ and $[n] := \{1,2, \dots, n\}$ for a positive integer $n$ with $n \geq 4$.
\begin{description}
	\item[Given:] $a_i \in \mathbf{R}$ for $i \in [n]$, $a_{ij} \in \overline{\mathbf{R}}$ for $1 \leq i < j \leq n$, and a positive integer $r$ with $2 \leq r \leq n-2$.
	\item[Question:] Is the quadratic function $f : \{0,1\}^n \rightarrow \overline{\mathbf{R}}$ defined by
	\begin{align}
	f(x_1,x_2,\dots,x_n) :=
	\begin{cases}
	\displaystyle
	\sum_{i \in [n]} a_i x_i + \sum_{1 \leq i < j \leq n}a_{ij}x_ix_j & \text{if $\displaystyle\sum_{i \in [n]} x_i = r$},\\
	+\infty & \text{otherwise}
	\end{cases}\label{eq:f}
	\end{align}
	M-convex?
\end{description}
Notice that if $r = 1$ or $r = n-1$,
then $f$ of the form~(\ref{eq:f}) is always a linear function.
Here we assume that $a_{ij} = a_{ji}$ for distinct $i,j \in [n]$ and $\textrm{dom }f$ is nonempty.
In this paper, functions can take the infinite value $+\infty$,
where $a < +\infty$, $a+\infty = +\infty$ for $a \in \mathbf{R}$,
and $0 \cdot (+\infty) = 0$.
In the case where $a_{ij}$ takes a finite value for all distinct $i,j \in [n]$,
the following theorem is immediate from {\cite[Theorem~6.4]{book/Murota03}} (see also \cite[Proposition~6.8]{book/Murota03}).
\begin{thm}[\cite{book/Murota03}; see also {\cite[Theorem~5.2]{AAM/MS04}}]\label{thm:poly finite}
	Suppose that $a_{ij}$ takes a finite value for all distinct $i,j \in [n]$.
	Then a function of the form~{\rm (\ref{eq:f})} is M-convex if and only if
	\begin{align}\label{eq:anti-tree metric}
	a_{ij} + a_{kl} \geq \min \{ a_{ik} + a_{jl}, a_{il}+a_{jk} \}
	\end{align}
	holds for every distinct $i,j,k,l \in [n]$.
\end{thm}
By Theorem~\ref{thm:poly finite},
if $a_{ij}$ is a finite value for all distinct $i,j \in [n]$,
then QMCTP is solvable in polynomial time.
However, if $a_{ij}$ can take the infinite value $+\infty$ for some distinct $i,j \in [n]$,
there exists an example such that the condition~(\ref{eq:anti-tree metric}) does not characterize M-convexity.
Indeed, define $f : \{0,1\}^5 \rightarrow \overline{\mathbf{R}}$ by
\begin{align}\label{eq:example}
f(x_1,x_2,x_3,x_4,x_5) :=
\begin{cases}
x_1x_3 + 2 x_1x_4 + (+\infty)\cdot x_1x_5 + x_3x_5 +2 x_4x_5 & \text{if $\sum_i x_i = 3$},\\
+\infty & \text{otherwise}.
\end{cases}
\end{align}
Then $f$ is M-convex; this can be verified by the definition of M-convexity.
However, the condition~(\ref{eq:anti-tree metric}) is violated since $a_{12} + a_{34} < \min \{ a_{13} + a_{24}, a_{14} + a_{23} \}$
with $a_{12} + a_{34} = 0$, $a_{13} + a_{24} = 1$, and $a_{14} + a_{23} = 2$.
Thus, in the general case, the complexity of QMCTP is not settled yet.

A quadratic function of the form~(\ref{eq:f}) with an infinite quadratic coefficient arises naturally from a binary VCSP function.
A binary VCSP function $F$ satisfying the joint winner property can be transformed to
a function represented as the sum two {\it special} M-convex functions of the form~(\ref{eq:f}) with an infinite quadratic coefficient.
This fact explains the polynomial-time solvability of $F$ (see~\cite{arxiv/IMZ17} for details).
The class of functions represented as the sum of two {\it general} quadratic M-convex functions corresponds to a new tractable class of binary VCSPs.
Thus we need to consider QMCTP in the general case
for the first step to identify such a new tractable class.

In this paper, we settle QMCTP by showing the following negative result.
\begin{thm}\label{thm:co-NP comp}
	{\rm QMCTP} is co-NP-complete.
\end{thm}
We also prove a positive result under the following natural condition.
\begin{description}
	\item[Condition A:] For any $i \in [n]$,
	there exists $x \in \textrm{dom }f$ with $x_i = 1$.
\end{description}
\begin{thm}\label{thm:poly}
	If {\rm Condition A} holds,
	{\rm QMCTP} is solvable in $O(n^2)$ time.
\end{thm}

Note that checking Condition~A is an NP-complete problem,
since it is almost equivalent to checking the existence of a stable set of size $r- 1$.

The remainder of this paper is organized as follows.
In Section~\ref{sec:co-NP comp},
we prove Theorem~\ref{thm:co-NP comp}
and introduce three types of functions.
In Section~\ref{sec:chara},
we present a characterization of M-convexity under Condition~A for three types.
This characterization implies Theorem~\ref{thm:poly finite}.
In Section~\ref{sec:testing},
we propose an $O(n^2)$-time algorithm for each type of QMCTP,
and prove the validity of these algorithms.
Thus, we show Theorem~\ref{thm:poly}.

\section{Co-NP-Completeness of QMCTP}\label{sec:co-NP comp}
In this section,
we show the co-NP-completeness of QMCTP in the general case.
In order to show Theorem~\ref{thm:co-NP comp},
we prepare some lemmas.

In the terminology of discrete convex analysis, a set $X \subseteq \{0,1\}^n$ is said to be {\it M-convex} if
for $x,y \in X$ and $i \in \textrm{supp}(x) \setminus \textrm{supp}(y)$, there exists $j \in {\rm supp}(y) \setminus \textrm{supp}(x)$ such that $x - \chi_i + \chi_j, y + \chi_i - \chi_j \in X$.
That is, an M-convex set $X$ is nothing but the base family of some matroid
if we identify a 0-1 vector with a subset of $[n]$.
Note that if $f$ is M-convex, then $\textrm{dom }f$ is M-convex.
\begin{lem}\label{lem:i,j in domf}
	Suppose that $f$ is a function of the form~{\rm (\ref{eq:f})} such that ${\rm dom}\ f$ is M-convex.
	For some distinct $i,j \in [n]$, assume that there exist $x, y \in {\rm dom}\ f$ with $x_i = 1$ and $y_j = 1$.
	Then, if $a_{ij} < +\infty$,
	there exists $z \in {\rm dom}\ f$ with $z_i = z_j = 1$.
\end{lem}
\begin{proof}
	Take $x, y \in \textrm{dom }f$ with $|\textrm{supp}(x) \setminus \textrm{supp}(y)|$ minimum satisfying $x_i = y_j = 1$.
	It suffices to show $|\textrm{supp}(x) \setminus \textrm{supp}(y)| = 0$.
	Suppose, to the contrary, that $|\textrm{supp}(x) \setminus \textrm{supp}(y)| > 0$.
	First we assume $|\textrm{supp}(x) \setminus \textrm{supp}(y)| \geq 2$.
	Then there exists $i' \neq i$ such that $i' \in \textrm{supp}(x) \setminus \textrm{supp}(y)$.
	By the M-convexity of $\textrm{dom }f$ for $x$, $y$, and $i'$,
	there exists $j' \in \textrm{supp}(y) \setminus \textrm{supp}(x)$ such that $x - \chi_{i'} + \chi_{j'} \in \textrm{dom }f$.
	If $j' = j$, then $x' := x - \chi_{i'} + \chi_{j}$ satisfies $x_i' = x_j' = 1$, a contradiction.
	If $j' \neq j$, then $x' := x - \chi_{i'} + \chi_{j'}$ satisfies
	$x'_i = y_j = 1$ and $|\textrm{supp}(x') \setminus \textrm{supp}(y)| < |\textrm{supp}(x) \setminus \textrm{supp}(y)|$.
	This is also a contradiction to the choice of $x$ and $y$.
	Hence we have $|\textrm{supp}(x) \setminus \textrm{supp}(y)| = 1 = |\textrm{supp}(y) \setminus \textrm{supp}(x)|$.
	
	Since $x,y \in \textrm{dom }f$, it holds that $a_{kl}, a_{ik}, a_{jk} < +\infty$ for any $k,l \in \textrm{supp}(x - \chi_i) (= \textrm{supp}(y - \chi_j))$.
	Moreover, we have $a_{ij} < +\infty$ by the assumption.
	Hence we obtain $z := x - \chi_k + \chi_j \in \textrm{dom }f$ for $k \in \textrm{supp}(x - \chi_i)$.
	Hence $z$ satisfies $z_i = z_j = 1$, a contradiction.
	Thus, we have $|\textrm{supp}(x) \setminus \textrm{supp}(y)| = 0$.
\end{proof}

For a function $f$ of the form~(\ref{eq:f}),
we define an undirected graph $G_{f} = ([n], E_{f})$
by $E_{f} := \{ \{i,j\} \mid i,j \in [n],\ i \neq j,\ a_{ij} = +\infty \}$.
Notice that Condition~A holds if and only if,
for each $i \in [n]$,
there is a stable set in $G_f$ of size $r$ containing $i$.

\begin{lem}\label{lem:dom f}
	Suppose that {\rm Condition A} holds.
	Then ${\rm dom}\ f$ is an M-convex set if and only if
	each connected component of $G_{f}$ is a complete graph.
\end{lem}
\begin{proof}
	(if part).
	Let $A_1, A_2, \dots, A_m$ be the connected components of $G_{f}$.
	Then $\textrm{dom }f$ is represented by $\textrm{dom }f = \{ x \in \{0,1\}^n \mid \sum_i x_i = r,\ |\textrm{supp}(x) \cap A_p| \leq 1 \text{ for all }p \in [m] \}$.
	Hence $\textrm{dom }f$ can be regarded as the base family of a partition matroid.
	This implies that $\textrm{dom }f$ is M-convex.
	
	(only-if part).
	We prove the contrapositive.
	Suppose that some connected component of $G_{f}$ is not complete.
	That is, there exist distinct $i,j,k \in [n]$ such that $\{i,j\}, \{j,k\} \in E_{f}$ and $\{i,k\} \not\in E_{f}$.
	By Condition~A, $a_{ik} < +\infty$, and Lemma~\ref{lem:i,j in domf},
	there exists $x \in \textrm{dom }f$ with $x_i = x_k = 1$.
	
	Take any $x,y \in \textrm{dom }f$ with $x_i = x_k = 1$ and $y_j = 1$.
	Since $a_{ij} = a_{jk} = +\infty$, we have $\textrm{supp}(x) \setminus \textrm{supp}(y) \supseteq \{i,k\}$.
	Then for all $j' \in \textrm{supp}(y) \setminus \textrm{supp}(x)$,
	it holds that $x - \chi_i + \chi_{j'} \not\in \textrm{dom }f$ or $y + \chi_i - \chi_{j'} \not\in \textrm{dom }f$.
	Indeed,
	if $j' = j$, then $x - \chi_i + \chi_j \not\in \textrm{dom }f$ holds from $a_{kj} = +\infty$,
	and if $j' \neq j$, then $y + \chi_i - \chi_{j'} \not\in \textrm{dom }f$ holds from $a_{ij} = +\infty$.
	This implies that $\textrm{dom }f$ is not M-convex.
\end{proof}

Here we consider the following problem (P),
which is the problem for testing the M-convexity of $\textrm{dom }f$:
\begin{description}
	\item[Given:] A graph $G = (V, E)$ having a stable set of cardinality $r$.
	\item[Question:] Let $T := \bigcup \{ S \subseteq V \mid S \text{ is a stable set of } G \text{ with } |S| = r \}$.
	Is each connected component of the subgraph of $G$ induced by $T$ a complete graph?
\end{description}

\begin{lem}\label{lem:(P)}
	The problem {\rm (P)} is co-NP-complete.
\end{lem}
\begin{proof}
	It is clear that the problem (P) is in co-NP.
	We show the co-NP-hardness of (P) by reduction from the stable set problem,
	which is an NP-complete problem:
	Given $G = (V, E)$ and a positive integer $k \leq |V|$,
	we determine whether $G$ contains a stable set of size at least $k$.
	For a given graph $G = (V,E)$ and a positive integer $m$, define $G_m := (V \cup V_m, E \cup E_m)$ by $|V_m| = m$, $V_m \cap V = \emptyset$, and $E_m := \{ \{i,j\} \mid i \in V,\ j \in V_m \}$.
	Let $T_m := \bigcup\{ S \subseteq V \cup V_m \mid S \text{ is a stable set of } G_m \text{ with } |S| = m \}$.
	Since $V_m$ is a stable set of $G_m$ satisfying $|V_m| = m$,
	we have $T_m \supseteq V_m$.
	If $T_m \supsetneq V_m$,
	the subgraph of $G_m$ induced by $T_m$ is not complete by the definition of $E_m$.
	Hence each connected component of the subgraph of $G_m$ induced by $T_m$ is complete if and only if
	$G$ does not have a stable set of cardinality at least $m$.
	Therefore we have the cardinality of a maximum stable set of $G$ by solving (P) for $G_k$ ($k = |V|, |V|-1, \dots, 1$).
	Indeed, the first $k$ such that we output ``no'' by solving (P) for $G_k$ is equal to the cardinality of a maximum stable set.
	Since the maximum stable set problem has a polynomial-time reduction to the complement of (P),
	(P) is co-NP-hard.
\end{proof}

We are now ready to prove Theorem~\ref{thm:co-NP comp}.
\begin{proof}[Proof of Theorem~\ref{thm:co-NP comp}]
	It is clear that QMCTP is in co-NP.
	We show the co-NP-hardness of QMCTP by reduction from the problem (P).
	Let $G = ([n], E)$ be a graph having a stable set of cardinality $r$.
	We define $f_G$ by
	\begin{align*}
	f_G(x_1,x_2,\dots,x_n) :=
	\begin{cases}
	\displaystyle
	\sum_{1 \leq i < j \leq n}a_{ij} x_i x_j & \text{if $\displaystyle\sum_{i \in [n]} x_i = r$},\\
	+\infty & \text{otherwise},
	\end{cases}
	\end{align*}
	where $a_{ij} := +\infty$ for $\{i,j\} \in E$ and $a_{ij} := 0$ for $\{i,j\} \not\in E$.
	Note that $x \in \textrm{dom }f$ if and only if $\textrm{supp}(x)$ is a stable set of $G$.
	We have $\textrm{dom }f_G \neq \emptyset$
	by the assumption that $G$ has a stable set of cardinality $r$.
	We define $X$ by $X := \bigcup \{ S \subseteq [n] \mid S \text{ is a stable set of } G \text{ of size } r \}$.
	Then there exists $x \in \textrm{dom }f_G$ with $x_i = 1$ if and only if $i \in X$.
	For $x \in \{0,1\}^X$,
	define $\tilde{x} \in \{0,1\}^n$ by $\tilde{x}_i := x_i$ if $i \in X$ and $\tilde{x}_i := 0$ if $i \in [n] \setminus X$.
	Moreover define $f_G|_X(x) := f_G(\tilde{x})$ for $x \in \{0,1\}^X$.
	By the definition of $X$,
	$f_G$ is M-convex (i.e., $\textrm{dom }f_G$ is M-convex) if and only if
	$f_G|_X$ is M-convex (i.e., $\textrm{dom }f_G|_X$ is M-convex).
	Furthermore, by Lemma~\ref{lem:dom f}, $f_G|_X$ is M-convex if and only if each connected component of the subgraph of $G$ induced by $X$ is complete.
	This means that we can solve (P) by solving QMCTP for $f_G$.
\end{proof}

\section{Characterization of Quadratic M-Convexity}\label{sec:chara}
In this section, we present a characterization of M-convexity under Condition A, which implies Theorem~\ref{thm:poly finite}.
By Lemma~\ref{lem:dom f},
we see that the following Condition B is necessary for the M-convexity.
\begin{description}
	\item[Condition B:] Each connected component of $G_{f}$ is a complete graph.
\end{description}
Therefore, in this section,
we can assume that a function $f$ of the form~(\ref{eq:f}) satisfies Conditions~A and B.
Let $A_1, A_2, \dots, A_m$ be the vertex sets of the connected components of $G_{f}$ of size at least two,
and define $A_0 := [n] \setminus \bigcup_{p = 1}^m A_p$,
which denotes the set of isolated vertices.
Then we classify the types of $f$ as follows.
\begin{description}
	\item[Type I:] $|A_0| + m \geq r+2$.
	\item[Type II:] $|A_0| + m = r+1$.
	\item[Type III:] $|A_0| + m = r$.
\end{description}
If $|A_0| + m < r$,
then we have $\textrm{dom }f = \emptyset$.
Hence we exclude this case.

\begin{thm}\label{thm:Type}
	Suppose that a function $f$ of the form~{\rm (\ref{eq:f})} satisfies {\rm Conditions~A} and~{\rm B}.
	Then the following hold.
	\begin{description}
		\item[(I):] $f$ of Type~{\rm I} is M-convex if and only if it holds that
		\begin{align}\label{eq:Type I}
		a_{ij} + a_{kl} \geq \min \{ a_{ik}+a_{jl}, a_{il}+a_{jk} \}
		\end{align}
		for every distinct $i,j,k,l \in [n]$.
		\item[(II):] $f$ of Type~{\rm II} is M-convex if and only if it holds that
		\begin{align}\label{eq:Type II}
		a_{ij} + a_{kl} = a_{il} + a_{jk}
		\end{align}
		for every $p \in [m]$, distinct $i,k \in A_p$, and distinct $j,l \in [n] \setminus A_p$.
		\item[(III):] $f$ of Type~{\rm III} is M-convex if and only if it holds that
		\begin{align}\label{eq:Type III}
		a_{ij} + a_{kl} = a_{il} + a_{jk}
		\end{align}
		for every distinct $p,q \in [m]$, distinct $i,k \in A_p$, and distinct $j,l \in A_q$.
	\end{description}
Moreover, if $f$ is an M-convex function of Type~{\rm II} or~{\rm III},
then $f$ is a linear function on ${\rm dom\ }f$,
that is,
there exist $p_i \in \mathbf{R}$ for each $i \in [n]$ and $\alpha \in \mathbf{R}$ satisfying
$f(x) = \sum_{i} p_i x_i + \alpha$ for any $x = (x_1, x_2, \dots, x_n) \in {\rm dom}\ f$.
\end{thm}
The function defined in~(\ref{eq:example}) is an example of Type~II.
If $a_{ij}$ is finite value for all distinct $i,j \in [n]$,
the function $f$ is of Type~I.
Hence Theorem~\ref{thm:Type} implies Theorem~\ref{thm:poly finite} as the finite case.
By Theorem~\ref{thm:Type},
we see that QMCTP is solvable in polynomial time under Conditions~A and~B.

In the proof of Theorem~\ref{thm:Type},
we use the following facts about the local exchange axiom  
characterizing M-convexity, which are immediate corollaries of~\cite[Theorem~6.4]{book/Murota03} (see also~\cite[Proposition~6.8]{book/Murota03}).
\begin{thm}[\cite{book/Murota03}]\label{thm:general M-conv}
	A function $f : \{0,1\}^n \rightarrow \overline{\mathbf{R}}$ with ${\rm dom}\ f \subseteq \{ x \in \{0,1\}^n \mid \sum_i x_i = r \}$ is M-convex if and only if
	${\rm dom}\ f$ is M-convex and
	\begin{align*}
	&f(z+\chi_i +\chi_j) + f(z +\chi_k+\chi_l)\\
	&\ \geq \min \{ f(z+\chi_i +\chi_k) + f(z +\chi_j+\chi_l), f(z+\chi_i +\chi_l) + f(z +\chi_j+\chi_k) \}
	\end{align*}
	holds for all $z \in \{0,1\}^n$ and all distinct $i,j,k,l \in [n]$ such that $z + \chi_i + \chi_j, z + \chi_k + \chi_l \in {\rm dom}\ f$.
\end{thm}

\begin{lem}\label{lem:quadratic M-conv.}
	A function $f$ of the form~{\rm (\ref{eq:f})} is M-convex if and only if
	for every distinct $i,j,k,l \in [n]$ such that there exists $z \in \{0,1\}^n$ with $z + \chi_i + \chi_j, z + \chi_k + \chi_l \in {\rm dom}\ f$,
	it holds that
	\begin{align*}
	a_{ij} + a_{kl} \geq \min \{ a_{ik} + a_{jl}, a_{il}+a_{jk} \}.
	\end{align*}
\end{lem}
Recall that, for a function $f$ of the form~(\ref{eq:f}),
$\textrm{dom }f$ is always M-convex,
since we assume that $f$ satisfies Conditions~A and~B.
By Lemma~\ref{lem:quadratic M-conv.},
the condition~(\ref{eq:anti-tree metric}) in Theorem~\ref{thm:poly finite} (or the condition~(\ref{eq:Type I}) in Theorem~\ref{thm:Type}) is sufficient for M-convexity.
However, this is not necessary in general.
\begin{proof}[Proof of Lemma~\ref{lem:quadratic M-conv.}]
	Take any $z \in \{0,1\}^n$ and distinct $i,j,k,l \in [n]$ such that $z + \chi_i + \chi_j, z + \chi_k + \chi_l \in {\rm dom}\ f$.
	By Theorem~\ref{thm:general M-conv}, it suffices to show that for such $i,j,k,l$,
	\begin{align*}
	&f(z + \chi_i + \chi_j) + f(z + \chi_k + \chi_l)\\
	&\ \geq \min\{f(z + \chi_i + \chi_k) + f(z + \chi_j + \chi_l), f(z + \chi_i + \chi_l) + f(z + \chi_j + \chi_k)\}
	\end{align*}
	holds if and only if $a_{ij} + a_{kl} \geq \min\{a_{ik} + a_{jl}, a_{il} + a_{jk}\}$ holds
	(note that the inequality $a_{ij} + a_{kl} \geq \min\{a_{ik} + a_{jl}, a_{il} + a_{jk}\}$ is independent of the choice of $z$).
	
	Define $g : \{0,1\}^n \rightarrow \overline{\mathbf{R}}$ by
	\begin{align*}
	g(x) := \sum_{i \in [n]} a_i x_i + \sum_{1 \leq i < j \leq n} a_{ij} x_i x_j
	\end{align*}
	for $x = (x_1, x_2, \dots, x_n) \in \{0,1\}^n$.
	Then we have
	\begin{align}
	f(z + \chi_i + \chi_j) = g(z) + a_i + a_j + \sum_{p \in \textrm{supp}(z)} a_{ip} + \sum_{p \in \textrm{supp}(z)} a_{jp} + a_{ij},\label{eq:2 M g(z+i+j)}\\
	f(z + \chi_k + \chi_l) = g(z) + a_k + a_l + \sum_{p \in \textrm{supp}(z)} a_{kp} + \sum_{p \in \textrm{supp}(z)} a_{lp} + a_{kl}\label{eq:2 M g(z+k+l)}.
	\end{align}
	Since $f(z + \chi_i + \chi_j)$ and $f(z + \chi_k + \chi_l)$ take finite values,
	each term of (\ref{eq:2 M g(z+i+j)}) and (\ref{eq:2 M g(z+k+l)}), i.e.,
	$g(z)$, $a_{ij}$, $a_{kl}$, and $a_{ip}, a_{jp}, a_{kp}, a_{lp}$ for $p \in \textrm{supp}(z)$, also takes a finite value.
	Hence we obtain
	\begin{align*}
	&f(z + \chi_i + \chi_j) + f(z + \chi_k + \chi_l)\\
	&\ \geq \min\{f(z + \chi_i + \chi_k) + f(z + \chi_j + \chi_l), f(z + \chi_i + \chi_l) + f(z + \chi_j + \chi_k)\}\\
	\Leftrightarrow\ &a_{ij} + a_{kl} \geq \min\{a_{ik} + a_{jl}, a_{il} + a_{jk}\}.
	\end{align*}
\end{proof}
For the function $f$ defined in~(\ref{eq:example}),
we can see that there is no $z \in \{0,1\}^5$ such that $z + \chi_1 + \chi_2$ and $z + \chi_3 + \chi_4$ both belong to $\textrm{dom }f$.
This is the reason why the inequality $a_{12} + a_{34} \geq \min \{ a_{13} + a_{24}, a_{14} + a_{23} \}$ is not necessary for the M-convexity of $f$.

A function $f$ is said to be {\it M-concave} if $-f$ is M-convex.
The following theorem (M-separation theorem) holds.
\begin{thm}[{\cite[Theorem~8.15]{book/Murota03}}]\label{thm:M-separation}
	Suppose that $f : \{0, 1\}^n \rightarrow \mathbf{R} \cup \{+\infty\}$ is M-convex and $g : \{ 0, 1\}^n \rightarrow \mathbf{R} \cup \{-\infty\}$ is M-concave satisfying ${\rm dom}\ f \cap {\rm dom}\ g \neq \emptyset$ and $g(x) \leq f(x)$ for any $x \in {\rm dom}\ f \cap {\rm dom}\ g$.
	Then there exist $\alpha^* \in \mathbf{R}$ and $p^* \in \mathbf{R}^n$ such that
	\begin{align*}
	g(x) \leq \alpha^* + \sum_{i \in [n]} p^*_i x_i \leq f(x) \qquad (x \in{\rm dom}\ f \cap {\rm dom}\ g).
	\end{align*}
\end{thm}

We are now ready to prove Theorem~\ref{thm:Type}.
\begin{proof}[Proof of Theorem~\ref{thm:Type}]
	First we show a characterization of M-convexity.
	For $i \in [n]$ denote by $B_i$ the connected component of $G_{f}$ containing $i$.
	That is, $B_i = \{i\}$ for $i \in A_0$, and $B_i = A_p$ for $i \in A_p$. 
	Note that $x \in \textrm{dom }f$ if and only if $\sum_i x_i = r$ and $|\textrm{supp}(x) \cap A_p| \leq 1$ for $p \in [m]$.
	If $a_{ij} = +\infty$ or $a_{kl} = +\infty$,
	then it holds that $f(z + \chi_i + \chi_j) = +\infty$ or $f(z + \chi_k + \chi_l) = +\infty$ for all $z \in \{0,1\}^n$.
	In the following, we consider each type in turn.
	
	\paragraph{Type I.}
	We show that for all distinct $i,j,k,l \in [n]$ with $a_{ij} < +\infty$ and $a_{kl} < +\infty$,
	there exists $z \in \{0,1\}^n$ such that $z + \chi_i + \chi_j, z + \chi_k + \chi_l \in \textrm{dom }f$.
	$|A_0 \setminus (B_i \cup B_j \cup B_k \cup B_l)| + |\{A_1, A_2,\dots, A_m\} \setminus \{B_i, B_j, B_k, B_l\}| \geq r-2$ holds since $|A_0| + m \geq r+2$.
	Therefore we can take $z \in \{0,1\}^n$ satisfying $\textrm{supp}(z) \subseteq [n] \setminus (B_i \cup B_j \cup B_k \cup B_l)$, $|\textrm{supp}(z) \cap A_p| \leq 1$ for $p \in [m]$, and $\sum_{i} z_i = r-2$.
	Then $z + \chi_i + \chi_j, z + \chi_k + \chi_l \in \textrm{dom }f$ holds for such $z$.
	
	By Lemma~\ref{lem:quadratic M-conv.},
	$f$ is M-convex if and only if for every distinct $i,j,k,l \in [n]$ with $a_{ij}, a_{kl} < +\infty$,
	it holds that $a_{ij} + a_{kl} \geq \min\{a_{ik} + a_{jl}, a_{il} + a_{jk}\}$.
	Moreover, if $a_{ij} = +\infty$ or $a_{kl} = +\infty$,
	then $a_{ij} + a_{kl} \geq \min\{a_{ik} + a_{jl}, a_{il} + a_{jk}\}$ automatically holds.
	Hence $f$ is M-convex if and only if for every distinct $i,j,k,l \in [n]$,
	it holds that $a_{ij} + a_{kl} \geq \min\{a_{ik} + a_{jl}, a_{il} + a_{jk}\}$.
	
	\paragraph{Type II.}
	We show that for distinct $i,j,k,l \in [n]$ with $a_{ij}, a_{kl} < +\infty$, there exists $z \in \{0,1\}^n$ such that $z + \chi_i + \chi_j, z + \chi_k + \chi_l \in \textrm{dom }f$
	if and only if $(B_i \cup B_j) \cap (B_k \cup B_l) \neq \emptyset$ holds
	(note that we have $B_i \cap B_j = B_k \cap B_l = \emptyset$ since $a_{ij}, a_{kl} < +\infty$).
	
	Suppose $(B_i \cup B_j) \cap (B_k \cup B_l) = \emptyset$ (i.e., $B_i$, $B_j$, $B_k$, and $B_l$ are all disjoint).
	Then $|A_0| + m \geq 4$.
	Hence $r \geq 3$ since $|A_0| + m = r+1$.
	Furthermore we obtain $|A_0 \setminus (B_i \cup B_j \cup B_k \cup B_l)| + |\{A_1, A_2,\dots, A_m\} \setminus \{B_i, B_j, B_k, B_l\}| = r-3$.
	Hence for all $z \in \{0,1\}^n$ such that $\sum_i z_i = r-2$ and $|\textrm{supp}(z) \cap A_p| \leq 1$ for $p \in [m]$ with $A_p \neq B_i, B_j,B_k,B_l$, it holds that $|\textrm{supp}(z) \cap (B_i \cup B_j \cup B_k \cup B_l)| \neq \emptyset$.
	This means $z + \chi_i+\chi_j \not\in \textrm{dom }f$ or $z+\chi_k+\chi_l \not\in \textrm{dom }f$.
	Thus, for $i,j,k,l \in [n]$ with $(B_i \cup B_j) \cap (B_k \cup B_l) = \emptyset$, there is no $z$ satisfying $z + \chi_i+\chi_j, z+\chi_k+\chi_l \in \textrm{dom }f$.
	
	Suppose $(B_i \cup B_j) \cap (B_k \cup B_l) \neq \emptyset$.
	Without loss of generality, we also suppose $B_i \cap B_k \neq \emptyset$.
	Then there exists $p \in [m]$ such that $B_i = B_k = A_p$.
	Since $|A_0| + m = r+1$,
	we have $|A_0 \setminus (B_i \cup B_j \cup B_k \cup B_l)| + |\{A_1, A_2,\dots, A_m\} \setminus \{B_i, B_j, B_k, B_l\}| = |A_0 \setminus (B_j \cup B_l)| + |\{A_1, A_2,\dots, A_m\} \setminus \{A_p, B_j, B_l\}| \geq r-2$.
	Therefore we can take $z \in \{0,1\}^n$ satisfying $\textrm{supp}(z) \subseteq [n] \setminus (B_i \cup B_j \cup B_k \cup B_l)$, $|\textrm{supp}(z) \cap A_p| \leq 1$ for $p \in [m]$, and $\sum_{i} z_i = r-2$.
	Then $z + \chi_i + \chi_j, z + \chi_k + \chi_l \in \textrm{dom }f$ holds for such $z$.
	
	By Lemma~\ref{lem:quadratic M-conv.},
	$f$ is M-convex if and only if for every $p \in [m]$, distinct $i,k \in A_p$, and distinct $j,l \in [n] \setminus A_p$,
	it holds that $a_{ij} + a_{kl} \geq \min\{a_{ik} + a_{jl}, a_{il} + a_{jk}\}$.
	Since $a_{ik} = +\infty$, the above inequality can be represented as $a_{ij} + a_{kl} \geq a_{il} + a_{jk}$.
	Moreover, by replacing $j$ with $l$,
	we have $a_{ij} + a_{kl} \leq a_{il} + a_{jk}$.
	Hence $f$ is M-convex if and only if for every $p \in [m]$, distinct $i,k \in A_p$, and distinct $j,l \in [n] \setminus A_p$,
	it holds that $a_{ij} + a_{kl} = a_{il} + a_{jk}$.
	
	\paragraph{Type III.}
	We show that for distinct $i,j,k,l \in [n]$ with $a_{ij}, a_{kl} < +\infty$, there exists $z \in \{0,1\}^n$ such that $z + \chi_i + \chi_j, z + \chi_k + \chi_l \in \textrm{dom }f$
	if and only if $B_i \cup B_j = B_k \cup B_l$ holds.
	
	Suppose $B_i \cup B_j \neq B_k \cup B_l$.
	Without loss of generality,
	we also suppose $B_i \neq B_k$ and $B_i \neq B_l$.
	Then $|A_0| + m \geq 3$.
	Hence $r \geq 3$ since $|A_0| + m = r$.
	Furthermore we obtain $|A_0 \setminus (B_i \cup B_j \cup B_k \cup B_l)| + |\{A_1, A_2,\dots, A_m\} \setminus \{B_i, B_j, B_k, B_l\}| \leq |A_0 \setminus (B_i \cup B_k \cup B_l)| + |\{A_1, A_2,\dots, A_m\} \setminus \{B_i, B_k, B_l\}|  = r-3$.
	Hence for all $z \in \{0,1\}^n$ such that $\sum_i z_i = r-2$ and $|\textrm{supp}(z) \cap A_p| \leq 1$ for $p \in [m]$ with $A_p \neq B_i, B_j,B_k,B_l$, it holds that $|\textrm{supp}(z) \cap (B_i \cup B_j \cup B_k \cup B_l)| \neq \emptyset$.
	This means $z + \chi_i+\chi_j \not\in \textrm{dom }f$ or $z+\chi_k+\chi_l \not\in \textrm{dom }f$.
	Thus, for $i,j,k,l \in [n]$ with $(B_i \cup B_j) \cap (B_k \cup B_l) = \emptyset$, there is no $z$ satisfying $z + \chi_i+\chi_j, z+\chi_k+\chi_l \in \textrm{dom }f$.
	
	Suppose $B_i \cup B_j = B_k \cup B_l$.
	Without loss of generality, we also suppose $B_i = B_k$ and $B_j = B_l$.
	Then there exist distinct $p, q \in [m]$ such that $B_i = B_k = A_p$ and $B_j = B_l = A_q$.
	Since $|A_0| + m = r$,
	we have $|A_0 \setminus (B_i \cup B_j \cup B_k \cup B_l)| + |\{A_1, A_2,\dots, A_m\} \setminus \{B_i, B_j, B_k, B_l\}| = |A_0| + |\{A_1, A_2,\dots, A_m\} \setminus \{A_p, A_q\}| = r-2$.
	Therefore we can take $z \in \{0,1\}^n$ satisfying $\textrm{supp}(z) \subseteq [n] \setminus (B_i \cup B_j \cup B_k \cup B_l)$, $|\textrm{supp}(z) \cap A_p| \leq 1$ for $p \in [m]$, and $\sum_{i} z_i = r-2$.
	Then $z + \chi_i + \chi_j, z + \chi_k + \chi_l \in \textrm{dom }f$ holds for such $z$.
	
	By Lemma~\ref{lem:quadratic M-conv.},
	$f$ is M-convex if and only if for every distinct $p,q \in [m]$, distinct $i,k \in A_p$, and distinct $j,l \in A_q$,
	it holds that $a_{ij} + a_{kl} \geq \min\{a_{ik} + a_{jl}, a_{il} + a_{jk}\}$.
	Since $a_{ik} = a_{jl} = +\infty$, the above inequality can be represented as $a_{ij} + a_{kl} \geq a_{il} + a_{jk}$.
	Moreover, by replacing $j$ with $l$,
	we have $a_{ij} + a_{kl} \leq a_{il} + a_{jk}$.
	Hence $f$ is M-convex if and only if for every distinct $p,q \in [m]$, distinct $i,k \in A_p$, and distinct $j,l \in A_q$,
	it holds that $a_{ij} + a_{kl} = a_{il} + a_{jk}$.
	
	\paragraph{Linearity.}
	Then we show linearity of an M-convex function $f$ of Type~II or~III.
	By the characterization of Type~II or~III,
	the function $g$ defined by
	\begin{align*}
	g(x) :=\begin{cases}
	f(x) & \text{if $f(x) < +\infty$},\\
	-\infty & \text{if $f(x) = +\infty$}
	\end{cases}
	\end{align*}
	is M-concave for an M-convex function $f$ of Type~II or~III.
	By Theorem~\ref{thm:M-separation},
	there exist $\alpha^* \in \mathbf{R}$ and $p^* \in \mathbf{R}^n$ such that
	\begin{align*}
	f(x) = g(x) \leq \alpha^* + \sum_{i \in [n]} p^*_i x_i \leq f(x) \qquad (x \in{\rm dom}\ f).
	\end{align*}
	This means that $f$ is a linear function on $\textrm{dom }f$.
\end{proof}

\section{Testing Quadratic M-Convexity in Quadratic Time}\label{sec:testing}
In this section,
we present an $O(n^2)$-time algorithm for QMCTP under the assumption that a function $f$ of the form~(\ref{eq:f}) satisfies Condition~A (and Condition~B).
By Theorem~\ref{thm:Type},
it suffices to give an $O(n^2)$-time algorithm for checking the condition (\ref{eq:Type I}), (\ref{eq:Type II}), or (\ref{eq:Type III}) in Theorem~\ref{thm:Type} for each type, respectively.

\subsection{Algorithms}
Our idea used in a proposed algorithm for Type~I is that
the quadratic coefficients $(a_{ij})_{i,j \in [n]}$ of input $f$ are transformed into another $(\hat{a}_{ij})_{i,j \in [n]}$
which has an easily checkable property if $(a_{ij})_{i,j \in [n]}$ satisfies (\ref{eq:Type I}).
For Types~II and~III,
we give simpler conditions equivalent to~(\ref{eq:Type II}) and~(\ref{eq:Type III}),
and check the new one.

We say that $(a_{ij})_{i,j \in [n]}$ satisfies the {\it anti-tree metric property} if $(a_{ij})_{i,j \in [n]}$ satisfies (\ref{eq:Type I}),
that is,
$a_{ij} + a_{kl} \geq \min \{ a_{ik} + a_{jl}, a_{il}+a_{jk} \}$ holds for all distinct $i,j,k,l \in [n]$.
We also say that $(a_{ij})_{i,j \in [n]}$ satisfies the {\it anti-ultrametric property} if $a_{ij} \geq \min \{ a_{ik}, a_{jk} \}$ holds for all distinct $i,j,k \in [n]$.

\begin{description}
	\item[Algorithm I (for Type I).]
	\item[Step 1:] Define $\alpha := \min\{ a_{ij} \mid i,j \in [n]\}$,
	$b_i := \min\{ a_{ij} \mid j \in [n] \setminus \{i\} \} - \alpha$ for $i \in [n]$,
	and $\hat{a}_{ij} := a_{ij}$ for distinct $i,j \in [n]$.
	\item[Step 2:] Update $\hat{a}_{ij} \leftarrow \hat{a}_{ij} - b_i - b_j$ for distinct $i,j \in [n]$.
	\item[Step 3:] If $(\hat{a}_{ij})_{i,j \in [n]}$ satisfies the anti-ultrametric property,
	output that ``$f$ is M-convex.''
	Otherwise, output that ``$f$ is not M-convex.''
	\qed
\end{description}

In Algorithms~II and~III, denote $A_p$ by $[n_p]$ for each $p \in [r]$,
where $n_p := |A_p|$.
\begin{description}
	\item[Algorithm II (for Type II).]
	\item[Step:]
	For all $p \in [r]$,
	if $a_{ij} + a_{i+1,j+1} = a_{i+1, j} + a_{i, j+1}$ holds for every $i \in [n_p-1]$ and $j \in \{ n_p+1, n_p+2, \dots, n-1 \}$,
	output that ``$f$ is M-convex.''
	Otherwise, output that ``$f$ is not M-convex.''
	\qed
\end{description}
\begin{description}
	\item[Algorithm III (for Type III).]
	\item[Step:] For all distinct $p, q \in [r]$,
	if $a_{ij} + a_{i+1,j+1} = a_{i+1, j} + a_{i, j+1}$ holds for every $i \in [n_p-1]$ and $j \in [n_q-1]$,
	output that ``$f$ is M-convex.''
	Otherwise, output that ``$f$ is not M-convex.''
	\qed
\end{description}

\begin{thm}\label{thm:Type I,II,III}
	Algorithms~{\rm I},~{\rm II}, and~{\rm III} work correctly and run in $O(n^2)$ time.
\end{thm}

We can check whether $f$ satisfies Condition~B,
i.e., each connected component of $G_f$ is a complete graph, in $O(n^2)$ time.
Thus, by Theorem~\ref{thm:Type I,II,III},
we obtain Theorem~\ref{thm:poly}.
In the rest of this section,
we give the proof of Theorem~\ref{thm:Type I,II,III}.
It is clear that the running time of Algorithms~II and~III are $O(n^2)$.
In Section~\ref{subsec:proof}, we show the validity of Algorithms~I,~II, and~III,
and show that we can check the anti-ultrametric property of given $(\hat{a}_{ij})_{i,j \in [n]}$ in $O(n^2)$ time in Step~3 of Algorithm~I.

\subsection{Proof of Theorem~\ref{thm:Type I,II,III}}\label{subsec:proof}
For brevity of notation,
we denote $\min \{ a_{ij} \mid j \in [n] \setminus \{i\} \}$ by $\min_{j} a_{ij}$.
\subsubsection{Validity of Algorithm~I.}
Observe that if $(a_{ij})_{i,j \in [n]}$ satisfies the anti-tree metric property,
so does $(\hat{a}_{ij})_{i,j \in [n]}$ defined by
\begin{align*}
\hat{a}_{ij} =
\begin{cases}
a_{ij} + b & \text{if $i^* \in \{i,j\}$},\\
a_{ij} & \text{if $i^* \not\in \{i,j\}$},
\end{cases}
\qquad (i,j \in [n],\ i \neq j)
\end{align*}
for some $i^* \in [n]$ and $b \in \mathbf{R}$.
This means that the (inverse) operation of Step~2 does not change the anti-tree metric property.
Furthermore, It is known that the anti-ultrametric property is stronger than the anti-tree metric property~\cite{JJIAM/HM04,book/SempleSteel03}.
Thus, if Algorithm~I returns the output ``$f$ is M-convex,''
then $(a_{ij})_{i,j \in [n]}$ satisfies the anti-tree metric property.
Therefore, the validity of Algorithm~I is established by proving that Algorithm~I returns ``$f$ is M-convex''
whenever $(a_{ij})_{i,j \in [n]}$ satisfies the anti-tree metric property.
We need some lemmas to show this statement.
In the following, suppose that $(a_{ij})_{i,j \in [n]}$ satisfies the anti-tree metric property.
Recall that $\alpha := \min\{ a_{ij} \mid i,j \in [n]\}$.

\begin{lem}\label{lem:Type I key}
	Suppose that $\min_{j} a_{ij} = \alpha$ holds for all $i \in [n]$.
	Then $(a_{ij})_{i,j \in [n]}$ satisfies the anti-ultrametric property.
\end{lem}
\begin{proof}
	Suppose, to the contrary, that there exist distinct $i,j,k \in [n]$ with $a_{ij} < \min\{a_{ik}, a_{jk}\}$.
	Then $a_{ik} > \alpha < a_{jk}$ holds.
	Hence, by the assumption,
	there exists $l \in [n] \setminus \{i,j,k\}$ satisfying $a_{kl} = \alpha$.
	Then we obtain $a_{ik} > a_{ij} < a_{jk}$ and $a_{jl} \geq a_{kl} \leq a_{il}$.
	Thus, for such $i,j,k,l \in [n]$,
	it holds that $a_{ik} + a_{jl} > a_{ij} + a_{kl} < a_{il} + a_{jk}$.
	This contradicts the anti-tree metric property of $(a_{ij})_{i,j \in [n]}$.
\end{proof}

By Lemma~\ref{lem:Type I key},
it suffices to show that if $(a_{ij})_{i,j \in [n]}$ satisfies the anti-tree metric property,
it holds that $\min_{j} \hat{a}_{ij} = \alpha$ for any $i \in [n]$ after Step~2 of Algorithm~I.
In the following, we prove this.

\begin{lem}\label{lem:Type I 1}
	Suppose that $\min_{j'} a_{ij'} = a_{ij} > \alpha$ holds
	for $i,j \in [n]$.
	Then there exists $k \in [n]$ such that $a_{jk} = \alpha$.
\end{lem}
\begin{proof}
	We show this by induction on $n$ (the number of variables of $f$).
	
	In the case of $n = 4$,
	it suffices to prove that if $\min \{a_{12}, a_{13}, a_{14}\} = a_{12} > \alpha$,
	we have $\min\{a_{23}, a_{24}\} = \alpha$.
	Suppose, to the contrary, that $\min\{a_{23}, a_{24}\} > \alpha$.
	By the assumption and $\min \{a_{12}, a_{13}, a_{14}, a_{23},a_{24},a_{34}\} = \alpha$,
	we obtain $a_{34} = \alpha$.
	Then, since $\min \{a_{12}, a_{13}, a_{14}\} = a_{12}$,
	we have $a_{14} \geq a_{12} \leq a_{13}$,
	and since $\min\{a_{23}, a_{24}\} > \alpha$,
	we have $a_{23} > a_{34} < a_{24}$.
	Therefore we have $a_{14} + a_{23} > a_{12} + a_{34} < a_{13}+a_{24}$.
	This contradicts the anti-tree metric property of $(a_{ij})_{i,j \in [n]}$.
	
	In the case of $n \geq 5$,
	it suffices to prove that if $\min\{a_{12}, a_{13}, \dots, a_{1n}\} = a_{12} > \alpha$,
	we have $\min\{a_{23}, a_{24}, \dots, a_{2n}\} = \alpha$.
	Suppose, to the contrary, that $\min\{a_{23}, a_{24}, \dots, a_{2n}\} = a_{23} > \alpha$.
	Since $(a_{ij})_{i,j \in [n]}$ defines an M-convex function,
	$(a_{ij})_{i,j \in \{2,3,\dots,n \}}$ also defines an M-convex function.
	Moreover, since $\min\{a_{12}, a_{13}, \dots, a_{1n}\} > \alpha$,
	we have $\min\{ a_{ij} \mid i,j \in \{2,3,\dots, n\} \} = \alpha$.
	By $\min\{a_{23}, a_{24}, \dots, a_{2n}\} = a_{23} > \alpha$ and the induction hypothesis,
	there exists $k \in \{2,3,\dots,n \}$ such that $a_{3k} = \alpha$ (without loss of generality assume $k = 4$).
	Since $\min\{a_{12}, a_{13}, \dots, a_{1n}\} = a_{12}$ and $\min\{a_{23}, a_{24}, \dots, a_{2n}\} > \alpha$,
	it holds that $a_{14} \geq a_{12} \leq a_{13}$ and $a_{23} > a_{34} < a_{24}$.
	Therefore we have $a_{14}+ a_{23} > a_{12} + a_{34} < a_{13} + a_{24}$.
	This contradicts the anti-tree metric property of $(a_{ij})_{i,j \in [n]}$.
\end{proof}

Define $G_{\min} := (V_{\min}, E_{\min})$ by $E_{\min} := \{\{i,j\} \mid a_{ij} = \alpha \}$ and $V_{\min} := \{ i \in [n] \mid i \in \exists e \in E_{\min} \}$.
\begin{lem}\label{lem:Type I 2}
	$G_{\min}$ is connected.
	Moreover, for all $i,j \in V_{\min}$,
	there exists an $i$-$j$ path having at most two edges in $G_{\min}$.
\end{lem}
\begin{proof}
	It suffices to prove that there exists an $i$-$j$ path having at most two edges in $G_{\min}$ for all distinct $i,j \in V_{\min}$.
	Suppose, to the contrary, that $G_{\min}$ does not have an $i$-$j$ path with at most two edges for some $i,j \in V_{\min}$.
	Hence we have $\{i,j\} \not\in E_{\min}$.
	Furthermore there exist $k,l \in V_{\min}$ such that $\{i,k\}, \{j,l\} \in E_{\min} \not\ni \{i,l\}, \{j,k\}$.
	For such $i,j,k,l$,
	we obtain $a_{ij} + a_{kl} > a_{ik}+a_{jl} < a_{il}+a_{jk}$, a contradiction.
\end{proof}

\begin{lem}\label{lem:Type I 3}
	Suppose that $\min_{j'} a_{ij'} = a_{ij} > \alpha$ holds
	for $i,j \in [n]$.
	Define $b_i := a_{ij} - \alpha$.
	Then $a_{ik} - b_i \geq \min_{i'} a_{i'k}$ holds for all $k \in [n] \setminus \{i\}$.
\end{lem}
\begin{proof}
	Take any $k \in [n] \setminus \{i\}$.
	If $\min_{i'} a_{i'k} = \alpha$,
	then it holds that $a_{ik} - b_i = a_{ik} - a_{ij} + \alpha \geq \alpha = \min_{i'} a_{i'k}$.
	This means that if $k = j$,
	the statement holds by Lemma~\ref{lem:Type I 1}.
	Hence, in the following, suppose $\min_{i'} a_{i'k} > \alpha$
	(note that $k \neq j$ holds).
	If $a_{ik} = \min_{i'} a_{i'k}$,
	we have $\min_{j'} a_{ij'} = \alpha$ by Lemma~\ref{lem:Type I 1}.
	This contradicts $\min_{j'} a_{ij'} = a_{ij} > \alpha$.
	Thus we obtain $a_{ik} > \min_{i'} a_{i'k}$.
	We consider the following three cases.
	
	(Case 1: $\min_{i'} a_{i'k} = a_{jk}$).
	By Lemma~\ref{lem:Type I 1},
	$a_{jl} = \alpha$ holds for some $l \in [n] \setminus \{i,j,k\}$.
	Since $(a_{ij})_{i,j \in [n]}$ satisfies the anti-tree metric property,
	we have $a_{ik} + a_{jl} \geq \min\{ a_{ij}+a_{kl}, a_{il}+a_{jk} \}$.
	Moreover, by $\min_{j'} a_{ij'} = a_{ij}$ and $\min_{i'} a_{i'k} = a_{jk}$,
	it holds that $a_{il} \geq a_{ij}$ and $a_{kl} \geq a_{jk}$, respectively.
	Hence $a_{ik} + a_{jl} \geq \min\{ a_{ij}+a_{kl}, a_{il}+a_{jk} \} \geq a_{ij} + a_{jk}$ holds.
	Since $a_{jl} = \alpha$ and $b_i = a_{ij} - \alpha$ hold,
	we obtain $a_{ik} - b_i \geq a_{jk} = \min_{i'} a_{i'k}$.
	
	(Case 2: $\min_{i'} a_{i'k} = a_{kl}$ for some $l \neq j$ and $a_{jl} = \alpha$).
	Since $(a_{ij})_{i,j \in [n]}$ satisfies the anti-tree metric property,
	we have $a_{ik} + a_{jl} \geq \min\{ a_{ij}+a_{kl}, a_{il}+a_{jk} \}$.
	Moreover, by $\min_{j'} a_{ij'} = a_{ij}$ and $\min_{i'} a_{i'k} = a_{kl}$,
	it holds that $a_{il} \geq a_{ij}$ and $a_{jk} \geq a_{kl}$, respectively.
	Hence $a_{ik} + a_{jl} \geq \min\{ a_{ij}+a_{kl}, a_{il}+a_{jk} \} \geq a_{ij} + a_{kl}$ holds.
	Since $a_{jl} = \alpha$ and $b_i = a_{ij} - \alpha$ hold,
	we obtain $a_{ik} - b_i \geq a_{kl} = \min_{i'} a_{i'k}$.
	
	(Case 3: $\min_{i'} a_{i'k} = a_{kl}$ for some $l \neq j$ and $a_{jl} > \alpha$).
	By Lemma~\ref{lem:Type I 1},
	we have $j, l \in V_{\min}$
	since $\min_{j'} a_{ij'} = a_{ij} > \alpha$ and
	$\min_{i'} a_{i'k} = a_{kl} > \alpha$.
	By Lemma~\ref{lem:Type I 2},
	there exists a $j$-$l$ path having at most two edges.
	Then the assumption of $a_{jl} > \alpha$ means that there exists $p \in [n] \setminus \{i,j,k,l\}$ such that $a_{jp} = a_{lp} = \alpha$.
	
	
	Since $(a_{ij})_{i,j \in [n]}$ satisfies the anti-tree metric property, it holds that $a_{ik} + a_{jp} \geq \min\{ a_{ip}+a_{jk}, a_{ij}+a_{kp} \}$.
	By $\min_{j'} a_{ij'} = a_{ij}$ and $\min_{i'} a_{i'k} = a_{kl}$,
	it holds that $a_{ip} \geq a_{ij}$ and $a_{kp} \geq a_{kl}$, respectively.
	Moreover, by $a_{jk} \geq a_{kl}$,
	we have $a_{ik} + a_{jp} \geq \min\{ a_{ip}+a_{jk}, a_{ij}+a_{kp} \} \geq a_{ij} + a_{kl}$.
	Since $a_{jp} = \alpha$ and $b_i = a_{ij} - \alpha$ hold,
	we obtain $a_{ik} - b_i \geq a_{kl} = \min_{i'} a_{i'k}$.
\end{proof}

Let $(\hat{a}_{ij})_{i,j \in [n]}$ be the quadratic coefficients after Step~2,
that is, $\hat{a}_{ij} = a_{ij} - b_i -b_j$ for distinct $i,j \in [n]$.
We show that $\min_{j} \hat{a}_{ij} = \alpha$ for any $i \in [n]$.

Take any $i \in [n]$.
Since $b_i \geq 0$,
it holds that $\min_j \hat{a}_{ij} \leq \min_j a_{ij} - b_i = \alpha$.
Next we prove $\min_j \hat{a}_{ij} \geq \alpha$.
If $\min_j a_{ij} = \alpha$,
we have $b_i = 0$.
Hence we obtain $\min_j \hat{a}_{ij} = \min_j \{a_{ij} - b_j\} \geq \alpha$.
If $\min_j a_{ij} > \alpha$,
by Lemma~\ref{lem:Type I 3},
we obtain $\min_j \hat{a}_{ij} \geq \min_j \{ \min_k a_{jk} - b_j \} = \alpha$.

\subsubsection{Time Complexity of Algorithm~I.}
It is clear that Steps~1 and~2 in Algorithm~I can be done in $O(n^2)$ time.
In the following,
we devise an $O(n^2)$-time algorithm for determining whether $(\hat{a}_{ij})_{i,j \in [n]}$ satisfies the anti-ultrametric property,
while a direct verification of checking the anti-ultrametric property in Step~3 takes $O(n^3)$ time.

First we present a key lemma for designing an $O(n^2)$-time algorithm.
This fact is well known~\cite{JJIAM/HM04, arxiv/IMZ17} and the point here is to allow the infinite value.
\begin{lem}[{\cite[Lemma 8]{arxiv/IMZ17}}]\label{lem:algo key}
	$(\hat{a}_{ij})_{i,j \in [n]}$ satisfies the anti-ultrametric property if and only if
	there exist some laminar family $\mathcal{L}$ on $[n]$ and some $c_U \in \overline{\mathbf{R}}$ for $U \in \mathcal{L}$ such that
	\begin{itemize}
		\item $[n]\in \mathcal{L}$,
		\item if $U \subsetneq U'$, then $c_U > c_{U'}$ holds,
		\item $\hat{a}_{ij} = c_{U(i,j)}$ holds for any distinct $i,j \in [n]$,
		where $U(i,j)$ is the minimal element in $\mathcal{L}$ including $\{i,j\}$.
	\end{itemize}
\end{lem}

By Lemma~\ref{lem:algo key}, we obtain the following natural procedure \texttt{Decompose},
which updates a laminar family $\mathcal{L}$ and defines $c_U \in \overline{\mathbf{R}}$ for $U \in \mathcal{L}$.
Suppose that we are given $U \subseteq [n]$ and $w \in \overline{\mathbf{R}}$.

\begin{description}
	\item[Procedure: $\texttt{Decompose}(U, w)$.]
	\item[Step 1:] If $|U| \leq 1$ or $w = +\infty$, then stop.
	\item[Step 2:] Take any $i \in U$.
	Define $e := \min \{ \hat{a}_{ij} \mid j \in U \setminus \{i\} \}$ and $X := \argmin \{ \hat{a}_{ij} \mid j \in U \setminus \{i\} \}$.
	\item[Step 3:] If $e > w$, then $\mathcal{L} \leftarrow \mathcal{L} \cup \{ U \}$, $c_U := e$, and $w \leftarrow e$.
	\item[Step 4:] Execute $\texttt{Decompose}(X, w)$ and $\texttt{Decompose}(U \setminus X, w)$.
	\qed
\end{description}
For initialization, let $\mathcal{L} := \{[n]\}$ and $c_{[n]} := \alpha$.
Observe that if $(\hat{a}_{ij})_{i,j \in [n]}$ satisfies the anti-ultrametric property,
$\texttt{Decompose}([n], \alpha)$ constructs an appropriate laminar family $\mathcal{L}$ and $c_U$ for $U \in \mathcal{L}$ corresponding to $(\hat{a}_{ij})_{i,j \in [n]}$.
Moreover $\texttt{Decompose}([n], \alpha)$ runs in $O(n^2)$ time.

We are ready to describe an algorithm for checking the anti-ultrametric property as follows.
\begin{description}
	\item[Algorithm~A (for checking the anti-ultrametric property).]
	\item[Step 1:] Define $\mathcal{L} := \{ [n] \}$ and $c_{[n]} := \alpha$.
	\item[Step 2:] Execute $\texttt{Decompose}([n], \alpha)$.
	\item[Step 3:] Make a copy of $\mathcal{L}$ and denote it by $\mathcal{L}'$,
	that is, $\mathcal{L}' := \{ U' \mid U \in \mathcal{L} \}$ (the base set of $\mathcal{L}'$ is also $[n]$).
	\item[Step 4:] While $\mathcal{L}' \neq \emptyset$, do the following:
	\begin{description}
		\item[Step 4-1:] Take any minimal element $U' \in \mathcal{L}'$.
		Define $a_{ij}' := c_U$ for $\{i,j\} \subseteq U$ and $\{i,j\} \cap U' \neq \emptyset$.
		\item[Step 4-2:] Let $U_+ \in \mathcal{L}$ be the minimal element in $\mathcal{L}$ with $U \subsetneq U_+$.
		Update $U_+' \leftarrow U_+' \setminus U$.
		\item[Step 4-3:] Update $\mathcal{L}' \leftarrow \mathcal{L}' \setminus U'$.
	\end{description}
	\item[Step 5:] If $(\hat{a}_{ij})_{i,j \in [n]} = (a_{ij}')_{i,j \in [n]}$,
	then output ``$(\hat{a}_{ij})_{i,j \in [n]}$ satisfies the anti-ultrametric property.''
	Otherwise, output ``$(\hat{a}_{ij})_{i,j \in [n]}$ does not satisfy the anti-ultrametric property.''
	\qed
\end{description}
In Step~4 of Algorithm~A,
note that we define the value of $a_{ij}'$ exactly once for every distinct $i,j \in [n]$.
Hence the time complexity of Step~4 is $O(n^2)$ time.
Thus, we see that Algorithm~A runs in $O(n^2)$ time.
By Lemma~\ref{lem:algo key}, the validity of Algorithm~A is clear.
Therefore we obtain the following theorem.
\begin{thm}\label{thm:Type I algo}
	Algorithm~{\rm A} works correctly and runs in $O(n^2)$ time.
\end{thm}
By Theorem~\ref{thm:Type I algo},
we can determine whether $(\hat{a}_{ij})_{i,j \in [n]}$ satisfies the anti-ultrametric property in $O(n^2)$ time.

\begin{remark}
	\upshape
	The procedure $\texttt{Decompose}$ has already been proposed in the preprint version of~\cite{JJIAM/HM04} in the context of M-convexity, and \cite{IPL/CR89, JTB/WSSB77} in the context of ultrametrics.
	However, these papers deal with a different case where $a_{ij}$ takes a finite value for all distinct $i,j \in [n]$
	and the effective domain is subset of general integral vectors.
\end{remark}

\begin{remark}
	\upshape
	We can devise another simpler $O(n^2)$-time algorithm for Type~I
	if $a_{ij}$ takes a finite value for all distinct $i,j \in [n]$.
	Indeed,
	for the finite case,
	it is known that $(a_{ij})_{i,j \in [n]}$ satisfies the anti-tree metric property if and only if
	$(\hat{a}_{ij})_{i,j \in [n-1]}$ satisfies the anti-ultrametric property,
	where $\hat{a}_{ij} := a_{ij} - a_{in} - a_{jn}$ for all distinct $i,j \in [n-1]$~(see e.g.,~\cite{book/SempleSteel03}).
	Hence it suffices to check the anti-ultrametric property of $(\hat{a}_{ij})_{i,j \in [n]}$.
	However, in the general case,
	the above algorithm does not work,
	since the relation between anti-tree metric property and anti-ultrametric property fails.
	For example,
	consider the case of $a_{15} = +\infty$, $a_{12} = a_{34} = 2$, $a_{13} = a_{24} = a_{45} = 1$, $a_{14} = a_{23} = a_{25} = a_{35} = 0$.
	Then $(a_{ij})_{i,j \in [5]}$ does not satisfy the anti-tree metric property since $a_{12} + a_{34} (= 4) > a_{13} + a_{24} ( = 2) > a_{14} + a_{23} (= 0)$.
	However, $(\hat{a}_{ij})_{i,j \in [4]}$ satisfies the anti-ultrametric property
	since $\hat{a}_{12} = \hat{a}_{13} = \hat{a}_{14} = -\infty$, $\hat{a}_{23} = \hat{a}_{24} = 0$, and $\hat{a}_{34} = 1$.
\end{remark}

\subsubsection{Validity of Algorithms~II and~III.}
Let $N$ and $M$ be positive integers.
It suffices to prove that $a_{ij} + a_{kl} = a_{il} + a_{kj}$ holds for every $i,k \in [N]$ with $i < k$ and $j,l \in [M]$ with $j < l$
if $a_{ij} + a_{i+1,j+1} = a_{i+1,j} + a_{i,j+1}$ holds for every $i \in [N-1]$ and $j \in [M-1]$.
We show this by induction on $(k-i) + (l-j)$.
For $s \geq 1$ and $t \geq 1$, take any $i,k \in [N]$ with $k = i+s$ and $j,l \in [M]$ with $l = j + t$.
The case $s + t = 2$ holds by the assumption.
Suppose $s + t \geq 3$.
Without loss of generality, $s \geq 2$.
By the induction hypothesis,
we have $a_{ij} + a_{k-1, l} = a_{il} + a_{k-1,j}$ and $a_{k-1,j} + a_{kl} = a_{k-1,l} + a_{kj}$.
Hence we obtain $a_{ij} + a_{kl} = a_{il} + a_{kj}$.
This completes the induction step.

\section*{Acknowledgments}
The author thanks Hiroshi Hirai and Kazuo Murota for careful reading and numerous helpful comments.
The author also thanks the referees for helpful comments.
In particular, Algorithms~II and~III (for Types~II and~III, respectively) are suggested by one of the referees.
This research was supported by JSPS Research Fellowship for Young Scientists.


\end{document}